\documentclass{amsart}
\usepackage{amsfonts,amssymb,amsmath,amsthm}
\usepackage{url}
\usepackage{enumerate}

\begin{document}

\urlstyle{sf}
\newtheorem{thm}{Theorem}[section]
\newtheorem{lemma}[thm]{Lemma}
\newtheorem{prop}[thm]{Proposition}
\newtheorem{cor}[thm]{Corollary}
\theoremstyle{definition}
\newtheorem{defn}[thm]{Definition}
\newtheorem{remark}[thm]{Remark}
\newtheorem{assumption}[thm]{Assumptions}
\numberwithin{equation}{section}

\author{Alan Thompson}
\address{Department of Mathematical and Statistical Sciences\\ 
University of Alberta\\ 
Edmonton, AB, Canada, T6G 2G1}
\email{amthomps@ualberta.ca}
\thanks{This work was partially supported by NSERC}


\newcommand{\Z}{\mathbb{Z}}
\newcommand{\C}{\mathbb{C}}
\newcommand{\Q}{\mathbb{Q}}
\newcommand{\Proj}{\mathbb{P}}
\newcommand{\calO}{\mathcal{O}}
\newcommand{\calL}{\mathcal{L}}
\newcommand{\calE}{\mathcal{E}}
\newcommand{\calT}{\mathcal{T}}
\newcommand{\calR}{\mathcal{R}}
\newcommand{\calM}{\mathcal{M}}
\newcommand{\calN}{\mathcal{N}}
\newcommand{\calA}{\mathcal{A}}
\newcommand{\calP}{\mathcal{P}}
\newcommand{\Sym}{\mathrm{Sym}}

\newcounter{list2}

\title[Explicit models for threefolds fibred by K3 surfaces of degree two]{Explicit models for threefolds fibred by K3 surfaces of degree two}

\begin{abstract} We consider threefolds that admit a fibration by K3 surfaces over a nonsingular curve, equipped with a divisorial sheaf that defines a polarisation of degree two on the general fibre. Under certain assumptions on the threefold we show that its relative log canonical model exists and can be explicitly reconstructed from a small set of data determined by the original fibration. Finally we prove a converse to the above statement: under certain assumptions, any such set of data determines a threefold that arises as the relative log canonical model of a threefold admitting a fibration by K3 surfaces of degree two. \end{abstract}
\maketitle

\section{Introduction}\label{intro}

The aim of this paper is to produce an explicit construction for the relative log canonical model of a threefold that admits a fibration by K3 surfaces of degree two.

The explicit construction of threefolds is a problem that has attracted significant interest in recent years, fuelled largely by open questions about mirror symmetry and the classification of Calabi-Yau threefolds. Motivated by general classification theory for algebraic varieties, a common method by which such threefolds are constructed is by way of a K3 fibration. Most approaches to date involve embedding such fibrations into a toric ambient space \cite{sfk3f}\cite{cyvpk3s}\cite{k3fcy31}\cite{tcyhfwk3h}, as this provides a setting under which many properties of the constructed threefolds can be easily calculated.

In this paper we find an alternative method by which K3-fibred threefolds may be constructed. We begin by restricting our attention to threefolds that admit fibrations by K3 surfaces equipped with a polarisation of degree two; the standard example of such a K3 surface is a double cover of $\Proj^2$ ramified over a sextic curve. These K3 surfaces can be thought of as higher dimensional analogues of hyperelliptic curves of genus two, which can be seen as double covers of $\Proj^1$ ramified over six points. Using this analogy, in this paper we generalise to higher dimensions a construction of Catanese and Pignatelli \cite{flgi} which produces an explicit model for a surface admitting a fibration by hyperelliptic curves of genus two. This construction should provide a more general way to construct K3-fibred threefolds than the toric embedding method, but remains explicit enough that many properties of the constructed threefolds can still be easily calculated (see \cite[Chapter 5]{mythesis}).

More specifically, given a threefold $X$ admitting a fibration by K3 surfaces $\pi\colon X \to S$ over a smooth curve $S$, along with a divisorial (i.e. rank one reflexive) sheaf $\calL$ inducing a polarisation of degree two on the general fibre, our aim is to construct a birational model for $(X,\pi,\calL)$ over $S$. The model we choose to construct is known as the \emph{relative log canonical model}, a birational model arising from the minimal model programme (see \cite[Section 3.8]{bgav}) that has a good explicit description and is unique in its birational equivalence class.

The construction proceeds as follows: starting with a threefold fibred by K3 surfaces of degree two $(X,\pi,\calL)$ over a curve $S$ that satisfies certain assumptions (\ref{mainass}), we begin by finding a $5$-tuple of data $(\calE_1,\tau,\xi,\calE_3^+,\beta)$ on $S$ (see Definition \ref{5tupledefn}) that is determined by $(X,\pi,\calL)$. We then show that from this $5$-tuple it is possible to explicitly reconstruct the \emph{relative log canonical algebra} $\calR(X,\pi,\calL)$ of $(X,\pi,\calL)$ (defined in Section \ref{relsect}), from which the relative log canonical model of $(X,\pi,\calL)$ is easily computed as $\mathbf{Proj}_S \calR(X,\pi,\calL)$.

The following is the main result of this paper (Theorem \ref{relmodthm}), which categorises the output of the above construction. It shows that any threefold fibred by K3 surfaces of degree two $(X,\pi,\calL)$ satisfying Assumptions \ref{mainass} determines a $5$-tuple that is \emph{admissible} (see Definition \ref{admissdefn}), from which its relative log canonical model can be explicitly reconstructed and, furthermore, that given an admissible $5$-tuple we may always find a threefold fibred by K3 surfaces of degree two that determines that $5$-tuple. Therefore, our result gives a complete description of the threefolds that can arise as the relative log canonical models of threefolds fibred by K3 surfaces of degree two, in terms of their associated admissible $5$-tuples, along with an explicit method to construct them.

\begin{thm} Fix a nonsingular complex curve $S$. Let $(X,\pi,\calL)$ be a threefold fibred by K3 surfaces of degree two over $S$ that satisfies Assumptions \ref{mainass}. Then the associated $5$-tuple of $(X,\pi,\calL)$ over $S$ is admissible.

Conversely, let $\calR$ be a sheaf of $\calO_S$-algebras defined by an admissible $5$-tuple $(\calE_1,\tau,\xi,\calE_3^+,\beta)$. Let $X = \mathbf{Proj}_S(\calR)$ and $\pi\colon X \to S$ be the natural projection. Then there is a canonically defined polarisation sheaf $\calL$ on $X$ that makes $(X,\pi,\calL)$ into a threefold fibred by K3 surfaces of degree two that satisfies Assumptions \ref{mainass} and, furthermore, $\mathbf{Proj}_S(\calR)$ is the relative log canonical algebra of $(X,\pi,\calL)$ over $S$ and $(\calE_1,\tau,\xi,\calE_3^+,\beta)$ is its associated $5$-tuple.\end{thm}

In order to prove this result we rely heavily on results of Catanese and Pignatelli \cite{flgi}, who perform an analogous construction for the relative canonical models of surfaces fibred by hyperelliptic curves of genus two. The closeness of the analogy between K3 surfaces of degree two and hyperelliptic curves of genus two means that several of the results in Sections \ref{relcanalg}, \ref{constructingR} and \ref{genres} of this paper correspond directly with results in \cite{flgi}, although many of the proofs have been substantially modified to work in the higher dimensional case. In particular, our main result (Theorem \ref{relmodthm}) should be seen as a higher-dimensional analogue of \cite[Theorem 4.13]{flgi}.

\proof[Notation] We briefly mention several pieces of notation that will be used throughout this paper. Firstly, let
$D$ be a Weil divisor on a normal variety $X$. Then $D$ determines a divisorial sheaf $\calO_X(D)$ on $X$. The $m$th reflexive power of $\calO_X(D)$ is defined to be $\calO_X(D)^{[m]} := (\calO_X(D)^{\otimes m})^{\vee\vee}$, where $\mbox{}^\vee$ denotes the dual sheaf, and agrees with the divisorial sheaf $\calO_X(mD)$. A more detailed discussion of the correspondence between Weil divisors and divisorial sheaves may be found in \cite[Appendix 1]{c3f}. 

Next let $f\colon X - \to Y$ be a birational map and let $D$ be a Weil divisor on $X$. Then we denote the exceptional set of $f$ by $\mathrm{Ex}(f)$, defined to be the set of points $x \in X$ such that $f^{-1}$ is not well-defined at $f(x)$. Finally, to avoid confusion with the direct image, we denote the strict transform of $D$ under $f$ by the nonstandard notation $f_+D$.

\section{The Relative Log Canonical Model of a Threefold Fibred by K3 Surfaces of Degree Two}\label{relsect}

The aim of this paper is to find an explicit method to construct the relative log canonical model of a threefold fibred by K3 surfaces of degree two. However, before doing this we should check that this model is always well-defined. Begin by fixing a nonsingular complex curve $S$, then define:

\begin{defn} A \emph{threefold fibred by K3 surfaces of degree two} over $S$ is a triple $(X,\pi,\calL)$ consisting of
\begin{itemize} 
\item A $\Q$-Gorenstein normal complex variety $X$ of dimension $3$,
\item A flat, projective, surjective morphism $\pi\colon X \to S$ with connected fibres, whose general fibre is a K3 surface with at worst Du Val singularities, and
\item A divisorial sheaf $\calL$ on $X$ with $\calL^{[m]}$ invertible for some $m > 0$, that induces a nef and big divisorial sheaf $\calL_s$ satisfying $\calL_s.\calL_s = 2$ on a general fibre $X_s$ of $\pi$.
\end{itemize}
\end{defn}

Given a threefold fibred by K3 surfaces of degree two $(X,\pi,\calL)$, the \emph{relative log canonical algebra} of $(X,\pi,\calL)$ is defined to be the $\calO_S$-algebra
\[ \calR(X,\pi,\calL) := \bigoplus_{n \geq 0} \pi_*((\omega_X \otimes \calL)^{[n]}).\]

Under the assumption that the relative log canonical algebra $\calR(X,\pi,\calL)$ is finitely generated as an $\calO_S$-algebra, the \emph{relative log canonical model} $X^c$ of $(X,\pi,\calL)$ over $S$, defined in \cite[Section 3.8]{bgav}, is well-defined and equal to
\[ X^c := \mathbf{Proj}_S \calR(X,\pi,\calL).\]
This model admits a natural morphism $\pi^c\colon X^c \to S$ and, furthermore, there is a birational map $\phi\colon X - \to X^c$ over $S$ satisfying $\mathrm{codim}\, \mathrm{Ex}(\phi^{-1}) \geq 2$.

The aim of this paper is to find an explicit method to construct the relative log canonical model of a threefold fibred by K3 surfaces of degree two. However, before we can do this, we must first make some assumptions on our threefold fibred by K3 surfaces of degree two.

\begin{assumption} \label{mainass} $(X,\pi,\calL)$ is a threefold fibred by K3 surfaces of degree two that satisfies the following assumptions:
\begin{enumerate}[(i)]
\item The divisorial sheaf $\calL$ is isomorphic to $\calO_X(H) \otimes \pi^*\calM$, where $H$ is a prime divisor on $X$ that is flat over $S$ and $\calM$ is an invertible sheaf on $S$;  
\item The log pair $(X,H)$ is canonical  (see \cite[Definition 2.34]{bgav});
\item The sheaf $\calL_s$ induced on a general fibre of $\pi\colon X \to S$ by $\calL$ is invertible and generated by its global sections; and
\item The localisation of the relative log canonical algebra $\calR(X,\pi,\calL)_s \otimes_{\calO_{S,s}} k(s)$ at any point $s \in S$ is isomorphic to one of
\begin{itemize}
\item (hyperelliptic case)
\[ \C[x_1,x_2,x_3,z]/(z^2 - f_6(x_i)),\]
where $\mathrm{deg}(x_i) = 1$ and $\mathrm{deg}(z) = 3$; or
\item (unigonal case)
\[ \C[x_1,x_2,x_3,y,z]/(z^2 - g_6(x_i,y),\ g_2(x_i)),\]
where $\mathrm{deg}(x_i)= 1$, $\mathrm{deg}(y)=2$, $\mathrm{deg}(z) = 3$ and $g_6(0,0,0,1) \neq 0$. 
\end{itemize}
\end{enumerate}
\end{assumption}

\subsection{Remarks on these Assumptions}

We use this subsection to briefly remark upon the reasons behind these assumptions and to discuss when they hold in certain special cases. However, the first thing that we should check is that, under the assumptions above, the relative log canonical model $X^c$ of $X$ is well-defined. This will follow if we can show that the relative log canonical algebra $\calR(X,\pi,\calL)$ is finitely generated as an $\calO_S$-algebra.

\begin{lemma} \label{fingen} Suppose that $(X,\pi,\calL)$ is a threefold fibred by K3 surfaces of degree two that satisfies Assumptions \ref{mainass}\textup{(}i\textup{)} and \ref{mainass}\textup{(}ii\textup{)}. Then the relative log canonical algebra $\calR(X,\pi,\calL)$ is finitely generated as an $\calO_S$-algebra. \end{lemma} 

\begin{proof} We begin by noting that, as $\calM$ is invertible, we have isomorphisms 
\[(\omega_X \otimes \calL)^{[n]} \cong (\omega_X \otimes \calO_X(H))^{[n]} \otimes \pi^*\calM^n\]
for all $n \geq 0$. So, by the projection formula, there are isomorphisms 
\[\pi_*((\omega_X \otimes \calL)^{[n]}) \cong \pi_*((\omega_X \otimes \calO_X(H))^{[n]})\otimes \calM^n\] 
for all $n \geq 0$. Thus the $\calO_S$-algebra $\calR(X,\pi,\calL)$ is just the twist of the $\calO_S$-algebra $\calR(X,\pi,\calO_X(H))$ by the invertible sheaf $\calM$, and so $\calR(X,\pi,\calL)$ is finitely generated as an $\calO_S$-algebra if and only if $\calR(X,\pi,\calO_X(H))$ is (note that this also implies that the corresponding relative log canonical models $\mathbf{Proj}_S\calR(X,\pi,\calL)$ and $\mathbf{Proj}_S\calR(X,\pi,\calO_X(H))$ are isomorphic, a fact that will be useful later).

Finally, as the log pair $(X,H)$ is canonical, finite generation of $\calR(X,\pi,\calO_X(H))$ follows from results of the log minimal model program for threefolds \cite[Theorem 3.14]{ilmmplcp}.\end{proof}

\begin{remark} We note that both the proof of this lemma and the construction in this paper work just as well when Assumption \ref{mainass}(ii) is replaced by the weaker assumption ``the log pair $(X,H)$ is log canonical''. However, under this weaker assumption the proof of the final Theorem \ref{relmodthm}, that describes the output of our construction, fails to hold.\end{remark}

With this in place, we will briefly discuss Assumption \ref{mainass}(i). Na\"{i}vely, one might expect to define a polarisation on a threefold fibred by K3 surfaces of degree two simply by specifying a prime divisor $H$ on $X$ that is flat over $S$ and that induces a polarisation of the required type on a general fibre. Indeed, Assumption \ref{mainass}(i) implies that such a divisor always exists, in the form of the divisor $H$, and it follows from the proof of Lemma \ref{fingen} that the relative log canonical models of $(X,\pi,\calL)$ and $(X,\pi,\calO_X(H))$ are isomorphic over $S$. However, when we come to construct the relative log canonical model for a threefold fibred by K3 surfaces of degree two, we find that our construction produces both the model threefold and a polarisation sheaf on it (see Theorem \ref{relmodthm}), and that this polarisation sheaf does not necessarily admit a flat section. To account for this, Assumption \ref{mainass}(i) allows the polarisation to be twisted by the inverse image of a divisor on $S$.

Next, we prove a result that will allow Assumptions \ref{mainass}(i) and \ref{mainass}(ii) to be checked locally on $S$.

\begin{prop} \label{Lflat} Let $(X,\pi,\calL)$ be a threefold fibred by K3 surfaces of degree two. Suppose that for every closed point $s \in S$ there exists an affine open set $U_s \subset S$ containing $s$ and a reduced, irreducible divisor $H_{U_s}$ defined by a section in $H^0(\pi^{-1}(U_s),\calL)$, such that $H_{U_s}$ is flat over $U_s$ and the log pair $(\pi^{-1}(U_s), H_{U_s})$ is canonical. Then Assumptions \ref{mainass}\textup{(}i\textup{)} and \ref{mainass}\textup{(}ii\textup{)} hold for $(X,\pi,\calL)$. \end{prop}

\begin{proof} Note that in order to prove the proposition, it suffices to show that we may find an invertible sheaf $\calM$ on $S$ such that $\calL \otimes \pi^*\calM^{-1} \cong \calO_{X}({H})$ for some prime divisor ${H}$ that is flat over $S$ and makes $(X,H)$ canonical.

To construct $\calM$, we begin by choosing some ample invertible sheaf $\calN$ on $S$. By the ampleness property, we may find an integer $m > 0$ such that $\pi_*{\calL} \otimes \calN^m$ is generated by its global sections. Furthermore, by the projection formula and the Leray spectral sequence, we have an isomorphism
\begin{equation} H^0(X, {\calL} \otimes {\pi}^*\calN^m) \cong H^0(S, {\pi}_*{\calL} \otimes \calN^m). \label{H0iso} \end{equation}
In particular, the space of sections $H^0(X, {\calL} \otimes {\pi}^*\calN^m)$ is nonempty. Let $D$ be an effective divisor defined by a general section in this space.

Now, we say that an irreducible component $D_i$ of an effective divisor $D$ is \emph{horizontal} if ${\pi}(D_i) = S$ and \emph{vertical} if ${\pi}(D_i)$ is a closed point in $S$. Let $D^h$ denote the sum of the horizontal components of $D$ and $D^v$ denote the sum of the vertical components. As ${\pi}$ is proper, the image of any irreducible component must be closed and connected, so any irreducible component of $D$ is either horizontal or vertical and $D = D^h + D^v$. Furthermore, $D^h$ and $D^v$ must be effective because $D$ is. 

Now let $s \in S$ be a point over which the fibre $X_s$ is reducible or non-reduced and let $V$ be any prime divisor in $\mathrm{Supp}(X_s)$. By assumption, there exists an affine neighbourhood $U_s$ of $s$ and a section in $H^0({\pi}^{-1}(U_s),{\calL}) \cong H^0({\pi}^{-1}(U_s),\calL \otimes {\pi}^*\calN^m)$ that does not vanish on $V$. So, since ${\pi}_*{\calL} \otimes \calN^m$ is generated by its global sections, using the isomorphism \eqref{H0iso} we find that there exists a global section of ${\calL} \otimes {\pi}^*\calN^m$ that does not vanish on $V$. Therefore, the natural injection
\[H^0(X, {\calL}(-V) \otimes {\pi}^*\calN^m) \longrightarrow H^0(X, {\calL} \otimes {\pi}^*\calN^m)\]
cannot be surjective, so its image is Zariski closed in $H^0(X, {\calL} \otimes {\pi}^*\calN^m)$ and $V$ does not appear in $D$. Repeating this argument for the (finitely many) other components of reducible or non-reduced fibres, we see that no components of such fibres appear in $D$.

Therefore, only components of reduced, irreducible fibres may appear in $D^v$. So $D^v$ must be a sum of fibres and as such can be written as the inverse image of an effective divisor $E$ on $S$. We have
\[\calO_X(D^h) \cong {\calL} \otimes {\pi}^*(\calN^m \otimes \calO_{S}(-E)).\]
Let $\calM = \calN^{-m} \otimes \calO_{S}(E)$ and $H = D^h$. In order to complete the proof of Proposition \ref{Lflat} we just need to show that $H$ is reduced, irreducible and flat over $S$, and $(X,H)$ is canonical.

We begin with flatness. Let $H_i$ denote a prime divisor in $\mathrm{Supp}(H)$. To show that $H_i$ is flat over $S$ (as a divisor), it suffices to show that $H_i$ is flat when considered as a scheme over $S$. As $S$ is a nonsingular curve, by \cite[Proposition III.9.7]{hart} this will follow if we can show that any associated point of $H_i$ maps to the generic point of $S$. But $H_i$ is reduced and irreducible, so its only associated point is the generic point, which maps to the generic point of $S$ as ${\pi}|_{H_i}$ is surjective. Thus every component $H_i$ of $H$ is flat over $S$, so $H$ must be also. 

Finally, we have to show that $H$ is reduced and irreducible and that $(X,H)$ is canonical. Pick a finite subcover $\mathcal{U}$ of $\{U_s|s \in S\}$. Then, since $\pi_*\calO_X(D)$ is generated by its global sections, using the isomorphism \eqref{H0iso} we see that $D$ may be chosen so that $H|_{\pi^{-1}(U_s)}$ is a general member of the linear system $|H_{U_s}|$ for each $U_s \in \mathcal{U}$. But such a member is reduced and irreducible by Bertini's theorem and the pair $({\pi^{-1}(U_s)},H|_{\pi^{-1}(U_s)})$ is canonical by \cite[Corollary 2.33]{bgav}. So $H$ must be reduced and irreducible and, as the canonical property can be checked locally, the pair $(X,H)$ is canonical. \end{proof}

We will devote the remainder of this subsection to a discussion of Assumptions \ref{mainass}(iii) and \ref{mainass}(iv), beginning with a lemma that will, in certain cases, allow us to check Assumption \ref{mainass}(iv) as a statement on the cohomology of the fibres.

\begin{lemma} \label{locallemma} Suppose that $(X,\pi,\calL)$ is a threefold fibred by K3 surfaces of degree two that satisfies Assumptions \ref{mainass}\textup{(}i\textup{)} and \ref{mainass}\textup{(}ii\textup{)}. Let $s \in S$ be any point and let $X_s$ denote the fibre of $\pi\colon X \to S$ over $s$. If $\calL$ and $(\omega_X\otimes\calL)$ are $\pi$-nef in a neighbourhood of $X_s$, then there is an isomorphism
\[ \calR(X,\pi,L)_s \otimes_{\calO_{S,s}} k(s) \cong \bigoplus_{n \geq 0} H^0(X_s, (\omega_X \otimes \calL)^{[n]}).\]
\end{lemma}
\begin{proof} It suffices to prove that the natural maps 
\[\pi_*((\omega_X \otimes \calL)^{[n]})_s \otimes_{\calO_{S,s}} k(s) \longrightarrow H^0(X_s, (\omega_X \otimes \calL)^{[n]})\] 
are isomorphisms for all $n > 0$. This will follow from the theorem on cohomology and base change if we can show that the higher direct images $R^i\pi_*\left((\omega_X \otimes \calL)^{[n]}\right)$ vanish in a neighbourhood of $s$ for all $i > 0$ and all $n > 0$. 

In order to show this note first that, since $H$ is effective and $(X,H)$ is canonical, by \cite[Corollary 2.35]{bgav} we have that $(X,0)$ is also canonical. Furthermore, as $\calL$ is $\pi$-nef in a neighbourhood of $X_s$ and has self-intersection number two on any fibre of $\pi$, it is also $\pi$-big in a neighbourhood of $X_s$. Using this, the vanishing of the higher direct images follows immediately by applying \cite[Theorem 1.2.5]{immp} and \cite[Remark 1.2.6]{immp} to the pair $(X,0)$ and the divisorial sheaves $(\omega_X \otimes \calL)^{[n]}$.
\end{proof}

We use this lemma to prove the next proposition, which serves to motivate Assumption \ref{mainass}(iv) by showing that it holds for a smooth generic fibre in a threefold fibred by K3 surfaces of degree two.

\begin{prop} \label{smoothprop} Let $(X,\pi,\calL)$ be a threefold fibred by K3 surfaces of degree two that satisfies Assumptions \ref{mainass}\textup{(}i\textup{)} and \ref{mainass}\textup{(}ii\textup{)}. Suppose that the generic fibre of $\pi\colon X \to S$ is smooth. Then Assumption \ref{mainass}\textup{(}iv\textup{)} holds for a general $s \in S$.\end{prop}
 
\begin{proof} Let $s \in S$ be a general point and let $X_s$ denote the fibre over $s$. Then, by definition, $\calL$ is $\pi$-nef in a neighbourhood of $X_s$, so we may apply Lemma \ref{locallemma} and adjunction to $X_s$ to obtain that 
\[\calR(X,\pi,\calL)_s \otimes_{\calO_{S,s}} k(s) \cong \bigoplus_{n \geq 0} H^0(X_s, \calL_s^n),\]
where $\calL_s$ denotes the invertible sheaf induced on $X_s$ by $\calL$. 

By assumption, we see that $X_s$ is a smooth K3 surface of degree two with polarisation $\calL_s$. Let $H_s$ be a divisor on $X_s$ defined by a section of $\calL_s$. \cite[Proposition 8]{fk3s} shows that $|H_s|$ is either base point free or has the form $|H_s| = |2E| +F$, where $E$ is a smooth elliptic curve and $F$ is a fixed rational $(-2)$-curve. In either case, the proof of \cite[Corollary 5]{fk3s} shows that $|2H_s|$ is base point free and the morphism to projective space defined by $|3H_s|$ is birational onto its image. Using this, the algebra $\bigoplus_{n \geq 0} H^0(X_s, \calL_s^n)$ is easily calculated using the Riemann-Roch theorem, to give the hyperelliptic case of Assumption \ref{mainass}(iii) when $|H_s|$ is base point free and the unigonal case of Assumption \ref{mainass}(iii) when $|H_s|$ has base points. \end{proof}

Our next result shows that, on a general fibre of a threefold fibred by K3 surfaces of degree two, Assumption \ref{mainass}(iii) implies Assumption \ref{mainass}(iv). In fact, we see that more than this is true: Assumption \ref{mainass}(iii) implies that the \emph{hyperelliptic case} of Assumption \ref{mainass}(iv) holds on a general fibre. This assumption is crucial to our construction: whilst we expect that a related construction should exist for threefolds fibred by K3 surfaces of degree two with unigonal general fibre, this is a subject for another paper.

\begin{prop} \label{hyperprop} Suppose that $(X,\pi,\calL)$ is a threefold fibred by K3 surfaces of degree two that satisfies Assumptions \ref{mainass}\textup{(}i\textup{)}, \ref{mainass}\textup{(}ii\textup{)} and \ref{mainass}\textup{(}iii\textup{)}. Then the hyperelliptic case of Assumption \ref{mainass}\textup{(}iv\textup{)} holds at a general point $s \in S$.
\end{prop}
\begin{proof} Let $s \in S$ be a general point and let $X_s$ denote the fibre over $s$. Then, exactly as in the proof of Proposition \ref{smoothprop}, we see that it is enough to study the algebra $\bigoplus_{n \geq 0} H^0(X_s, \calL_s^n)$, where $\calL_s$ denotes the invertible sheaf induced on $X_s$ by $\calL$.  

By definition, $X_s$ is a K3 surface of degree two with polarisation $\calL_s$, that may have Du Val singularities. Let $f\colon Y \to X_s$ be a minimal resolution of $X_s$, where $Y$ is a smooth K3 surface. The inverse image $f^*\calL_s$ is an invertible sheaf that defines a polarisation of degree two on $Y$ and $H^0(Y, f^*\calL_s^n)\cong H^0(X_s, \calL_s^n)$ for all $n\geq0$. Given this, we argue in the same way as in the proof of Proposition \ref{smoothprop} on $Y$, whilst noting that Assumption \ref{mainass}(iii) implies that the linear system defined on $Y$ by $f^*\calL_s$ is base point free, so only the hyperelliptic case may occur.\end{proof}

The last result of this subsection is the strongest. It shows that if $X$ is smooth and $\pi\colon X \to S$ is semistable, then Assumption \ref{mainass}(iv) holds at \emph{every} point $s \in S$. 

\begin{thm} Suppose that $(X,\pi,\calL)$ is a threefold fibred by K3 surfaces of degree two satisfying Assumptions \ref{mainass}\textup{(}i\textup{)} and \ref{mainass}\textup{(}ii\textup{)}, and suppose further that $X$ is smooth and $\pi\colon X \to S$ is semistable \textup{(}i.e. all fibres of $\pi$ are reduced and have simple normal crossings\textup{)}. Then Assumption \ref{mainass}\textup{(}iv\textup{)} holds for $(X,\pi,\calL)$.
\end{thm}
\begin{proof} Applying the results of \cite[Section 2]{dk3sd2} locally around the degenerate fibres of $\pi$, we see that we may find a threefold fibred by K3 surfaces of degree two $(X',\pi',\calL')$ that is birational to $(X,\pi,\calL)$ over $S$ and has the same relative log canonical algebra, but for which $\calL'$ is $\pi'$-nef and $\omega_{X'}$ is trivial in a neighbourhood of any fibre (this construction is detailed in full in \cite[Section 2.4]{mythesis}). Given this, applying \cite[Theorem 3.1]{dk3sd2} in a neighbourhood of every fibre shows that Assumption \ref{mainass}(iv) holds for $(X',\pi',\calL')$ and so, as $(X',\pi',\calL')$ and $(X,\pi,\calL)$ have the same relative log canonical algebra, it must also hold for $(X,\pi,\calL)$.
\end{proof}

In light of this result, we conjecture that Assumption \ref{mainass}(iv) should hold for any threefold fibred by K3 surfaces of degree two that satisfies Assumptions \ref{mainass}(i) and \ref{mainass}(ii). An analogous result is known to hold for genus two curves, this was proved by Mendes-Lopes \cite[Theorem 3.7]{ml}.

\section{Structure of the Relative Log Canonical Algebra} \label{relcanalg}

We now embark upon our construction of the relative log canonical model of a threefold fibred by K3 surfaces of degree two. In order to do this, we will try to emulate the construction of the relative canonical model for a fibration by genus 2 curves, given originally by Catanese and Pignatelli \cite{flgi}. As such, the course of our construction will follow \cite{flgi} quite closely.

We begin by recalling the set up. Fix a nonsingular complex curve $S$, and let $(X,\pi,\calL)$ be a threefold fibred by K3 surfaces of degree two over $S$ satisfying Assumptions \ref{mainass}. Then it follows from Lemma \ref{fingen} that the relative log canonical model $X^c$ of $X$ over $S$ is well-defined, and the fibres of $\pi^c\colon X^c \to S$ are classified by Assumption \ref{mainass}(iv).

By definition, the general fibre $X_s$ of $\pi\colon X \to S$ is a (possibly singular) K3 surface of degree two with polarisation $\calL_s$ induced by $\calL$. Furthermore, by Assumption \ref{mainass}(iii), $\calL_s$ defines a base point free linear system on $X_s$, so the restriction of the birational map $\phi\colon X - \to X^c$ to $X_s$ is a birational morphism. It follows from Proposition \ref{hyperprop} that the image of $X_s$ under this morphism is a double cover of $\Proj^2$ ramified over a (possibly singular) sextic curve.

We are now ready to start our pursuit of an explicit construction for the relative log canonical model of $(X,\pi,\calL)$. Recall from Section \ref{relsect} that the relative log canonical algebra of $(X,\pi,\calL)$ is defined to be the graded algebra 
\[ \calR(X,\pi,\calL) = \bigoplus_{n=0}^{\infty} \calE_n := \bigoplus_{n=0}^{\infty} \pi_*((\omega_X \otimes \calL)^{[n]}) \]
and the relative log canonical model of $(X,\pi,\calL)$ is $X^c:=\mathbf{Proj}_S (\calR(X,\pi,\calL))$. We will try to find a way to construct $\calR(X,\pi,\calL)$ explicitly, which will in turn allow us to construct the relative log canonical model. First, however, we would like to know more about the structure of $\calR(X,\pi,\calL)$. 

Firstly, by definition, the sheaves $(\omega_X \otimes \calL)^{[n]}$ are reflexive for all $n \geq 0$. So, by \cite[Corollary 1.7]{srs}, the sheaves $\calE_n := \pi_*((\omega_X \otimes \calL)^{[n]})$ are also reflexive and thus, since $S$ is a smooth curve, must be locally free $\calO_S$-modules by \cite[Corollary 1.4]{srs}. 

Next, since the general fibre of $\pi\colon X \to S$ admits a birational morphism to a double cover of $\Proj^2$, there exists a birational involution $\iota$ on $X$ exchanging the sheets of this cover. We can use this involution to split the relative log canonical algebra into an invariant and an anti-invariant part. Let $U' \subset S$ be an open set. Then $U := \pi^{-1}(U')$ is \makebox{$\iota$-invariant} and $\iota$ acts linearly on the space of sections $H^0(U,(\omega_X \otimes \calL)^{[n]}) = \calE_n(U)$, which splits as the direct sum of the $(+1)$-eigenspace and the $(-1)$-eigenspace. 

This allows us to decompose $\calE_n$ into
\[\calE_n = \calE_n^+ \oplus \calE_n^-\]
and we can split the relative log canonical algebra as
\[\calR(X,\pi,\calL) = \calR(X,\pi,\calL)^+ \oplus \calR(X,\pi,\calL)^-.\]
Furthermore, observe that $\calR(X,\pi,\calL)^+$ is a subalgebra of $\calR(X,\pi,\calL)$, and that $\calR(X,\pi,\calL)^-$ is an $\calR(X,\pi,\calL)^+$-module.

This decomposition will prove to be invaluable when we attempt to construct $\calR(X,\pi,\calL)$. We can calculate the ranks of the locally free sheaves $\calE_n^+$ and $\calE_n^-$ for $n\geq 1$ to get the following table: 
\[ \renewcommand{\tabcolsep}{10mm} \renewcommand{\arraystretch}{1.5} \begin{tabular}{|c|c|c|}  \hline 
$n$ & rank $\calE_n^+$ &  rank $\calE_n^-$ \\ \hline
even & $\frac{(n+1)(n+2)}{2}$ & $\frac{(n-1)(n-2)}{2}$\\
odd & $\frac{(n-1)(n-2)}{2}$ & $\frac{(n+1)(n+2)}{2}$ \\ \hline
\end{tabular}\]
Furthermore, we know that $\calE_0 = \calE_0^+ = \calO_S$ and $\calE_1 = \calE_1^-$.

Next, we would like to study the multiplicative structure of $\calR(X,\pi,\calL)$, paying particular attention to how it interacts with the decomposition above. So let $\mu_{n,m}\colon \calE_n \otimes \calE_m \to \calE_{n+m}$ and $\sigma_n\colon \Sym^n(\calE_1) \to \calE_n$ denote the homomorphisms induced by multiplication in $\calR(X,\pi,\calL)$. The maps $\sigma_n$ will prove to be particularly useful as, if we can determine more information about them, we should be able to use them to reconstruct the sheaves $\calE_n$ from $\calE_1$. We have:

\begin{lemma}\label{sigmaprop} The maps $\sigma_n\colon \Sym^n(\calE_1) \to \calE_n$ are injective for all $n \geq 1$ and their image is contained in $\calE_n^+$ when $n$ is even and in $\calE_n^-$ when $n$ is odd.\end{lemma}
\begin{proof} We begin by showing injectivity. As $\Sym^n(\calE_1)$ and $\calE_n$ are locally free for all $n > 0$, it is enough to show that $\sigma_n$ is injective on the fibres of the associated vector bundles. But this follows easily from the explicit description of these fibres given by Assumption \ref{mainass}(iv). With this in place, the statement on the images of $\sigma_n$ follows immediately from the fact that $\calE_1 = \calE_1^-$.
\end{proof}  

Define $\calT_n := \mathrm{coker}(\sigma_n)$ and, using Lemma \ref{sigmaprop}, write
\begin{align*}
\calT_n^+ := \mathrm{coker}(\Sym^n(\calE_1) \to \calE_n^+) &\ \,\, \mathrm{for} \ n \ \mathrm{even}\\
\calT_n^- := \mathrm{coker}(\Sym^n(\calE_1) \to \calE_n^-) &\ \,\, \mathrm{for} \ n \ \mathrm{odd}.
\end{align*}
Then, by Lemma \ref{sigmaprop} again, we can decompose
\[\calT_n = \left\{ \begin{array}{ll}	\calT_n^+ \oplus \calE_n^- & \mathrm{for} \ n \ \mathrm{even}\\
										\calT_n^- \oplus \calE_n^+ & \mathrm{for} \ n \ \mathrm{odd}. \end{array} \right. \]
Finally, note that the sheaves $\calT^{\pm}_n$ are torsion sheaves.

With this in place, we are ready to begin describing how to construct $\calR(X,\pi,\calL)$. 

\section{Constructing the Relative Log Canonical Algebra} \label{constructingR}

In this section we detail the explicit construction of the relative log canonical algebra $\calR(X,\pi,\calL)$. In order to do this we follow the construction given by Catanese and Pignatelli in \cite{flgi}. This will involve constructing a graded subalgebra $\calA$ of $\calR(X,\pi,\calL)$ that is simpler to construct explicitly, and that can act as a ``stepping stone'' on the way to the construction of $\calR(X,\pi,\calL)$.

Before we start, however, it is convenient to explain some of the geometry that motivates this algebraic approach. As we mentioned before, it follows from Proposition \ref{hyperprop} that the general fibre of $\pi\colon X \to S$ admits a birational morphism to a double cover of $\Proj^2$ ramified over a sextic curve. In a similar fashion to Horikawa's \cite{oaspcg2} construction of models for surfaces fibred by genus two curves, one might consider constructing a na\"{i}ve model for $X$ as a double cover of a $\Proj^2$-bundle on $S$ ramified over a divisor that intersects the general fibre in a sextic. However, it is shown in \cite[Example 2.1.1]{mythesis} that if $\pi\colon X \to S$ contains any unigonal fibres then their structure is destroyed by such a construction.

To explain how we will solve this problem, we need to be a little more precise about what is going wrong. Let $S_0$ be the subset of $S$ over which the hyperelliptic case of Assumption \ref{mainass}(iv) holds, which is Zariski open by Proposition \ref{hyperprop}. The open set $\pi^{-1}(S_0)$ in $X$ is isomorphic to a double cover of the $\Proj^2$-bundle on $S_0$ given by  $\Proj_{S_0}(\calE_1)$. The branch divisor is defined using the cokernel of the map $\sigma_3\colon \Sym^3(\calE_1) \to \calE_3$, which is locally free on $S_0$. Unfortunately, we find that if we try to extend this definition to all of $S$ then we lose the local freeness of the cokernel, so the branch divisor is no longer well-defined. This can be solved by performing the construction on $S_0$ and extending to the whole of $S$ using the properness of the Hilbert scheme. However, this process may destroy the structure of the fibres over $S - S_0$. We refer the interested reader to \cite[Sections 1.3 and 2.1]{mythesis} for more details.

This problem only occurs on fibres where the cokernel of the map $\sigma_3$ is not locally free. As we saw above, this cokernel can be written as $(\calT_3^- \oplus \calE_3^+)$, where $\calT_3^-$ is a torsion sheaf. Furthermore, as we shall see in Lemma \ref{Tprops} below, $\calT_3^-$ is supported exactly on the points of $S$ corresponding to the unigonal fibres. This explains why this construction fails on such fibres.

To solve this problem, we will construct an algebra $\calA$ that takes better account of the properties of the maps $\sigma_n$ than $\Sym(\calE_1)$ does. Instead of a $\Proj^2$-bundle, $\mathbf{Proj}_S(\calA)$ will be a fibration of $S$ by rational surfaces. We can then try to construct $X^c = \mathbf{Proj}_S(\calR(X,\pi,\calL))$ as a double cover of $\mathbf{Proj}_S(\calA)$.

However, in order to do this we will need to better understand the maps $\sigma_n$. We begin by studying the structure of the cokernels $\calT_n$. We have the following analogue of \cite[Lemma 4.1]{flgi}:

\begin{lemma} \label{Tprops} Let $(X,\pi,\calL)$ be a threefold fibred by K3 surfaces of degree two over $S$ that satisfies Assumptions \ref{mainass}. Then
\begin{enumerate}[\textup{(}i\textup{)}]
\item $\calT_2 = \calT_2^+$ is isomorphic to the structure sheaf of an effective divisor $\tau$, supported on the points of $S$ corresponding to the unigonal fibres of $\pi$;
\item $\tau$ determines all the sheaves $\calT_n$ as follows:
\begin{align*}
\calT_{2n}^+ & \cong   \bigoplus_{i=1}^{n} \calO_{i\tau}^{\oplus(4(n-i)+1)}\\
\calT_{2n+1}^- & \cong  \bigoplus_{i=1}^n \calO_{i\tau}^{\oplus(4(n-i)+3)}
\end{align*}
\end{enumerate}
\end{lemma}
\begin{proof} (Following the proof of \cite[Lemma 4.1]{flgi}). Assumption \ref{mainass}(iv) describes the two possibilities (hyperelliptic and unigonal) for the localisation of the relative log canonical algebra $\calR(X,\pi,\calL)_s \otimes_{\calO_{S,s}} k(s)$ at any point $s \in S$. In the notation of that assumption, we see that $x_i$ are the $\iota$-antiinvariant sections and $y$ and $z$ are $\iota$-invariant. Furthermore, examination of the two cases shows that the cokernels $\calT_n$ are locally free away from the points where the unigonal case holds, so the torsion sheaves $\calT^+_{2n}$ and $\calT^-_{2n+1}$ are supported on these points.

Thus, we may restrict our attention to those points where the unigonal case holds. Around such a point $P$, the sheaf $\calE^+_2$ is locally generated by the sections $x_1^2$,~$x_2^2$, $x_3^2$, $x_1x_2$, $x_1x_3$, $x_2x_3$ and $y$. Furthermore, as $g_6(0,0,0,1) \neq 0$ at such a point, we may assume that the coefficient of $y^3$ in $g_6$ is non-zero and, by completing the square in the $x_i$, we may also assume that
\[g_2(x_i) = x_1^2 - x_2(ax_2 + bx_3),\]
for some $a,b \in \C$ not both zero. Then, by flatness, if $t$ is a uniformising parameter for $\calO_{S,P}$ we can lift the relation $g_2$ to
\[g_2(t) = x_1^2 - x_2(ax_2 + bx_3) + t \mu(t)y + t \psi(x_i,t).\]

Note that $\mu(t)$ is not identically zero, as $x_1$, $x_2$ and $x_3$ are algebraically independent for $t \neq 0$. Therefore, after changing coordinates in $S$, we may assume that $\mu(t) = t^{r-1}$ for a suitable integer $r \geq 1$. We call $r$ the \emph{multiplicity} of the point $P$. Using this and the relation above, the stalk of $\calT_2$ at $P$ is the $\calO_{S,P}$-module
\[ \calT_{2,P} = (\mathrm{coker}(\sigma_2))_P \cong \calE^+_{2,P} / \mathrm{Im}(\sigma_{2,P}) \cong \calO_{S,P}/(t^r),\]
generated by the class of $y$.

Define $\tau$ to be the divisor on $S$ given by $\sum_i r_i P_i$, where $P_i$ are the points in $S$ over which the fibres are unigonal and the $r_i$ are the corresponding multiplicities. Then the stalk of $\calO_{\tau}$ at $P_i$ is given by $\calO_{\tau,P_i} \cong \calO_{S,P_i}/(t^{r_i})$ and thus $\calT_2 \cong \calO_{\tau}$. This proves part (i) of Lemma \ref{Tprops}.

Next, we can also choose a lifting of $g_6$ of the form
\[g_6(t) = z^2 - g_6'(x_i,y,t).\]
Since $g_6$ is $\iota$-invariant, $g_6(t)$ must be also, otherwise $z$ would vanish identically on the fibre over $P$. By flatness, $g_2(t)$ and $g_6(t)$ are all the relations of the stalk of $\calR(X,\pi,\calL)$ at $P$. 

Now consider $\calT^-_{2n+1}$. Its stalk at $P$ is given by
\[ (\calT^-_{2n+1})_P = (\mathrm{coker}(\sigma_{2n+1}))_P \cong \calE^-_{2n+1,P}/\mathrm{Im}(\sigma_{2n+1,P}).\]
$\calE^-_{2n+1,P}$ is generated by the $2n^2 + 5n + 3$ monomials 
\[\{x_1h_{2n}(x_2,x_3,y),\ h_{2n+1}(x_2,x_3,y)\},\] where $h_i(x_2,x_3,y)$ denotes any monomial of degree $i$ in $x_2$, $x_3$ and $y$. Similarly, $\mathrm{Im}(\sigma_{2n+1,P})$ is generated by the $4n+3$ monomials 
\[\{x_1h_{2n}(x_2,x_3),\ h_{2n+1}(x_2,x_3)\}.\]
So $(\calT^-_{2n+1})_P$ is generated by the $2n^2 + n$ monomials
\[ \{x_1yh_{2n-2}(x_2,x_3,y),\ yh_{2n-1}(x_2,x_3,y)\}.\]
These monomials can be listed as
\begin{align*} \{y^n x_i\} & \ \mathrm{generates} \ \calO_{n\tau,P}^{\oplus 3} \\
\{x_1y^{n-1}h_2(x_2,x_3),\ y^{n-1}h_3(x_2,x_3)\} &\ \mathrm{generates} \ \calO_{(n-1)\tau,P}^{\oplus 7} \\
\{x_1y^{n-2}h_4(x_2,x_3),\ y^{n-2}h_5(x_2,x_3)\} &\ \mathrm{generates} \ \calO_{(n-2)\tau,P}^{\oplus 11} \\
& \vdots & \\
\{x_1yh_{2n-2}(x_2,x_3),\ yh_{2n-1}(x_2,x_3)\} &\ \mathrm{generates} \ \calO_{\tau,P}^{\oplus (4n - 1)}
\end{align*}
and we see that $\calT_{2n+1}^-  \cong \bigoplus_{i=1}^n \calO_{i\tau}^{\oplus(4(n-i)+3)}$.

Finally, a similar calculation gives $\calT_{2n}^+  \cong \bigoplus_{i=1}^n \calO_{i\tau}^{\oplus(4(n-i)+1)}$; full details may be found in the proof of \cite[Lemma 4.2.1]{mythesis}. This completes the proof of Lemma \ref{Tprops}.\end{proof}

Using Lemma \ref{Tprops}, if we know $\calT_2$ we can determine all of the cokernels $\calT_n$. So it seems sensible to expect that the structure of $\calR(X,\pi,\calL)$ might be determined by its structure in low degrees. With this in mind, we define:

\begin{defn} Let $\calA$ be the graded subalgebra of $\calR(X,\pi,\calL)$ generated by $\calE_1$ and $\calE_2$. Let $\calA_n$ denote its graded part of degree $n$ and write
\[ \calA = \calA_{\mathrm{even}} \oplus \calA_{\mathrm{odd}} = \Big(\bigoplus_{n=0}^{\infty} \calA_{2n}\Big) \oplus \Big(\bigoplus_{n=0}^{\infty} \calA_{2n+1}\Big).\]
\end{defn}

We similarly decompose $\calR(X,\pi,\calL) = \calR_{\mathrm{even}} \oplus \calR_{\mathrm{odd}}$. Then we have the following analogue of \cite[Lemma 4.3]{flgi}:

\begin{lemma} \label{Rprops} $\calR(X,\pi,\calL)$ is isomorphic to $\calA \oplus (\calA[-3] \otimes \calE_3^+)$ as a graded $\calA$-module. Furthermore, $\calA_{\mathrm{even}}$ is the $\iota$-invariant part of $\calR_{\mathrm{even}}$ and $\calA_{\mathrm{odd}}$ is the $\iota$-antiinvariant part of $\calR_{\mathrm{odd}}$. \end{lemma}
\begin{proof} (Following the proof of \cite[Lemma 4.3]{flgi}). We can unify the hyperelliptic and unigonal cases from Assumption \ref{mainass}(iv) by writing the localisation of the relative log canonical algebra $\calR(X,\pi,\calL)_s \otimes_{\calO_{S,s}} k(s)$ to a point over which the fibre is hyperelliptic as
\[ \C[x_1,x_2,x_3,y,z]/(y,\ z^2 - f_6(x_i)),\]
where the $x_i$ are $\iota$-antiinvariant of degree 1, and $y$ and $z$ are $\iota$-invariant with degrees 2 and 3 respectively. Then in both cases the stalk of $\calA$ is the subalgebra generated by $x_1$, $x_2$, $x_3$ and $y$, so $\calA_{\mathrm{even}}$ is $\iota$-invariant and $\calA_{\mathrm{odd}}$ is $\iota$-antiinvariant.

In both cases, locally on $S$ we may write 
\[\calR(X,\pi,\calL) \cong \calO_S[x_1,x_2,x_3,y,z]/(f_2(t), f_6(t))\] 
with $f_6(t) = z^2 - f_6'(x_i,y,t)$, so locally we have
\begin{enumerate}[(i)]
\item $\calA \cong \calO_S[x_1,x_2,x_3,y]/(f_2(t))$, and
\item $\calR(X,\pi,\calL) \cong \calA \oplus z\calA$.
\end{enumerate}
As $z$ is a local generator of $\calE_3^+$, this gives $\calR(X,\pi,\calL) \cong \calA \oplus (\calA[-3] \otimes \calE_3^+)$.

Finally, the statement on the $\iota$-invariant and $\iota$-antiinvariant parts follows from the fact that $\calA_{\mathrm{even}}$ is $\iota$-invariant and $\calR_{\mathrm{even}} \cong \calA_{\mathrm{even}} \oplus z\calA_{\mathrm{odd}}$, and  $\calA_{\mathrm{odd}}$ is $\iota$-anti\-invariant and $\calR_{\mathrm{odd}} \cong \calA_{\mathrm{odd}} \oplus z\calA_{\mathrm{even}}$.\end{proof}

$\mathbf{Proj}_S(\calA)$ is a fibration of $S$ by rational surfaces with natural projection map $\pi_{\calA}\colon \mathbf{Proj}_S(\calA) \to S$. The inclusion $\calA \subset \calR(X,\pi,\calL)$ yields a factorisation of the fibration $\pi\colon X \to S$ as
\[X \stackrel{\phi}{- \to} X^c = \mathbf{Proj}_S(\calR(X,\pi,\calL)) \stackrel{\psi}{\longrightarrow} \mathbf{Proj}_S(\calA) \stackrel{\pi_{\calA}}{\longrightarrow} S.\]
We will attempt to construct $\calA$ first, then use the properties of the map $\psi$ to reconstruct $\calR(X,\pi,\calL)$.

As $\calA$ is generated by $\calE_1$ and $\calE_2$, we might expect that $\calA$ can be reconstructed from the locally free sheaves $\calE_1$ and $\calE_2$ and the map $\sigma_2$ that relates them. The next proposition, our analogue of \cite[Lemma 4.4]{flgi}, gives us a way to do this:

\begin{prop} \label{Aseqs} With notation as above, there are exact sequences
\[\begin{array}{lcl} (*) & \Sym^2(\calE_1 \wedge \calE_1) \otimes \Sym^{n-2}(\calE_2) \stackrel{i_n}{\longrightarrow} \Sym^n(\calE_2) \longrightarrow \calA_{2n} \longrightarrow 0 & (n \geq 2) \\
(**) & \calE_1 \otimes (\calE_1 \wedge \calE_1) \otimes \calA_{2n-2} \stackrel{j_n}{\longrightarrow} \calE_1 \otimes \calA_{2n} \longrightarrow \calA_{2n+1} \longrightarrow 0 & (n \geq 1) \end{array}\]
where
\begin{align*} i_n\big((x_i \wedge x_j)(x_k \wedge x_l) \otimes r\big) & := \big(\sigma_2 (x_i x_k) \sigma_2 (x_j x_l) - \sigma_2 (x_i x_l)\sigma_2 (x_j x_k)\big)r,\\
j_n\big(l \otimes (x_i \wedge x_j) \otimes r\big) & := x_i \otimes \big(\sigma_2 (x_j l)r) - x_j \otimes (\sigma_2 (x_i l)r\big).
\end{align*}
Furthermore, if $n=2$ then sequence $(*)$ is also exact on the left.
\end{prop}
\begin{proof} (Based upon the proof of \cite[Lemma 4.4]{flgi}). The maps $\Sym^n(\calE_2) \to \calA_{2n}$ and $\calE_1 \otimes \calA_{2n} \to~\calA_{2n+1}$, induced by the ring structure of $\calA$, are surjective because $\calA$ is generated in degrees $\leq 2$ by definition. Since $\calE_n$ and $\calA_n$ are locally free, the respective kernels are locally free also. Furthermore, both sequences are complexes, by virtue of associativity and commutativity in $\calR(X,\pi,\calL)$.

It remains to show that $(*)$ and $(**)$ are exact in the middle. Since the kernels of the maps to $\calA_n$ are locally free, it is enough to prove this on the fibres of the associated vector bundles.  

We begin with sequence $(*)$. Suppose that $f$ is contained in the kernel of the map to $\calA_{2n}$. We wish to show that $f$ is also in the image of $i_n$.

If the fibre of $\pi\colon X \to S$ over the point under consideration is hyperelliptic, then $\calE_2$ is generated by the images $\sigma_2(x_ix_j)$ for all $i,j \in \{1,2,3\}$. Express $f$ in terms of these generators. Then perform the following algorithm on $f$:
\begin{enumerate}[(i)]
\item If any monomial of $f$ contains a factor of $\sigma_2(x_1x_i)\sigma_2(x_1x_j)$, with $i,j \in \{2,3\}$, replace this factor with $\sigma_2(x_1^2)\sigma_2(x_ix_j)$. Repeat this step until it terminates.
\item If any monomial of $f$ contains a factor of $\sigma_2(x_1x_3)\sigma_2(x_2x_i)$, with $i \in \{2,3\}$, replace this factor with $\sigma_2(x_1x_2)\sigma_2(x_3x_i)$.
\item If any monomial of $f$ contains a factor of $\sigma_2(x_2x_3)\sigma_2(x_2x_3)$, replace this factor with $\sigma_2(x_2^2)\sigma_2(x_3^2)$. Repeat this step until it terminates.
\item Collect like terms in $f$ and simplify.
\end{enumerate}
Call the result $f'$. Note that the kernel of the map to $\calA_{2n}$ is closed under these operations, so $f'$ is in this kernel. Furthermore, $\mathrm{Im}(i_n)$ is also closed under these operations and their inverses, so $f \in \mathrm{Im}(i_n)$ if and only if $f' \in \mathrm{Im}(i_n)$.

Now, any monomial in $f'$ must have the form
\[\sigma_2(x_1^2)^{n_{1,1}}\sigma_2(x_1x_2)^{n_{1,2}}\sigma_2(x_2^2)^{n_{2,2}}\sigma_2(x_2x_3)^{n_{2,3}}\sigma_2(x_3^2)^{n_{3,3}}\sigma_2(x_1x_3)^{n_{1,3}},\]
with $n_{1,2}, n_{2,3}, n_{1,3} \in \{0,1\}$ and $n_{1,3} = 1$ only if $n_{1,2} = n_{2,2} = n_{2,3} = 0$. However, under the map to $\calA_{2n}$ there are no relations between monomials of this form so, since $f'$ is in the kernel of this map, $f'$ must be the zero polynomial. But $0 \in \mathrm{Im}(i_n)$, so $f \in \mathrm{Im}(i_n)$ also.

The proof for points corresponding to unigonal fibres is very similar. This time, $\calE_2$ is generated by $y$ and the images $\sigma_2(x_ix_j)$ for all $i,j \in \{1,2,3\}$ with the exception of $(i,j) = (1,1)$. We perform the same set of operations on $f$, but with step (i) replaced by
\begin{enumerate}[(i)]
\item[(i')] If any monomial of $f$ contains a factor of $\sigma_2(x_1x_i)\sigma_2(x_1x_j)$, with $i,j \in \{2,3\}$, replace this factor with $(a\sigma_2(x_2^2) + b \sigma_2(x_2x_3))\sigma_2(x_ix_j)$, where the degree 2 relation in the unigonal fibre is given by $q(x_1,x_2,x_3) = x_1^2 - x_2(ax_2 + bx_3) = 0$ for some $a,b \in \C$. Repeat this step until it terminates.
\end{enumerate}
Any monomial in the resulting $f'$ must have the form
\[y^{n_{0}}\sigma_2(x_1x_2)^{n_{1,2}}\sigma_2(x_2^2)^{n_{2,2}}\sigma_2(x_2x_3)^{n_{2,3}}\sigma_2(x_3^2)^{n_{3,3}}\sigma_2(x_1x_3)^{n_{1,3}},\]
with $n_{1,2}, n_{2,3}, n_{1,3} \in \{0,1\}$ and $n_{1,3} = 1$ only if $n_{1,2} = n_{2,2} = n_{2,3} = 0$. With this, the remainder of the proof proceeds exactly as in the hyperelliptic case.

It remains to show that this sequence is exact on the left when $n = 2$. This will again follow from the corresponding statement on the fibres of the associated vector bundles. As the map induced by $i_2$ on the fibres of the associated vector bundles is linear, in order to prove that it is injective we need only show that the dimension (as a complex vector space) of its domain is equal to that of its image. A simple calculation yields that the dimension of a fibre of $\calA_4$ is 21, and the dimension of a fibre of $\Sym^2(\calE_2)$ is 15. So, as sequence $(*)$ is exact in the middle, the image of $i_2$ has dimension 6. But a fibre of $\Sym^2(\calE_1 \wedge \calE_1)$ also has dimension 6. Hence, $i_2$ is injective and sequence $(*)$ is exact on the left when $n = 2$. 

Next we consider sequence $(**)$. Given $f$ contained in the kernel of the map to $\calA_{2n+1}$, we wish to show that $f$ is contained in the image of $j_n$.

First consider the case where the fibre of $\pi\colon X \to S$ over the point under consideration is hyperelliptic. Then the fibre of the $\calO_S$-algebra $\calA$ over this point is isomorphic to  $\C[x_1,x_2,x_3]$. Since the $x_i$ form a basis for the fibre of $\calE_1$, we may write $f$ as 
\[f = x_1 \otimes f_1 + x_2 \otimes f_2 + x_3 \otimes f_3\]
for some $f_1,f_2,f_3 \in \C[x_1,x_2,x_3]$ of degree $2n$. This maps to $x_1 f_1 + x_2 f_2 + x_3 f_3$ under the map to $\calA_{2n+1}$, so the condition that $f$ is in the kernel of this map is equivalent to $x_1 f_1 + x_2 f_2 + x_3 f_3 = 0$.

Using this equation, we have $x_1|(x_2 f_2 + x_3 f_3)$. This implies that $f_2$ and $f_3$ have the form
\begin{align*} f_2 & = x_1 \, r_2(x_1,x_2,x_3) + x_3 \, s_{23}(x_2,x_3),\\
f_3 & = x_1 \, r_3(x_1,x_2,x_3) - x_2 \, s_{23}(x_2,x_3),
\end{align*} 
for $r_i, s_{ij} \in \C[x_1,x_2,x_3]$ of degree $(2n-1)$. Repeating this process for $x_2$ and $x_3$, we get
\begin{align*} f_1 & = x_2x_3 \, r_1(x_1,x_2,x_3) + x_2 \, s_{12}(x_1,x_2) + x_3 \, s_{13}(x_1,x_3),\\
f_2 & = x_1x_3 \, r_2(x_1,x_2,x_3) - x_1 \, s_{12}(x_1,x_2) + x_3 \, s_{23}(x_2,x_3), \\
f_3 & = x_1x_2 \, r_3(x_1,x_2,x_3) - x_1 \, s_{13}(x_1,x_3) - x_2 \, s_{23}(x_2,x_3),
\end{align*}
for $r_i, s_{ij} \in \C[x_1,x_2,x_3]$ of degrees $(2n-2)$ and $(2n-1)$ respectively. Furthermore, as $f$ is in the kernel of the map to $\calA_{2n+1}$, we must have $r_1 + r_2 + r_3 = 0$.

Let $l_{ij}(x_i,x_j)$ be any linear factor of $s_{ij}(x_i,x_j)$. Using this, we can express $s_{ij}(x_i,x_j) = l_{ij}(x_i,x_j) \, s'_{ij}(x_i,x_j)$ for $s'_{ij} \in \C[x_1,x_2,x_3]$ of degree $(2n-2)$. Then we have
\begin{align*} f = &\, x_1 \otimes (x_2x_3 \, r_1 + x_2 \, l_{12} \, s'_{12} + x_3 \, l_{13} \, s'_{13})\, + \\
&\,  x_2 \otimes (x_1x_3 \, r_2 - x_1 \, l_{12} \, s'_{12} + x_3 \, l_{23} \, s'_{23}) \,+ \\ 
&\, x_3 \otimes (- x_1x_2 \, r_1 - x_1x_2 \, r_2 - x_1 \, l_{13} \, s'_{13} - x_2 \,l_{23} \, s'_{23}) \\
 = &\, j_n\big( x_2 \otimes (x_1 \wedge x_3) \otimes r_1 + x_1 \otimes (x_2 \wedge x_3) \otimes r_2 + l_{12} \otimes (x_1 \wedge x_2) \otimes s'_{12} +\\ 
&\, \makebox[1em]{} l_{13} \otimes (x_1 \wedge x_3) \otimes s'_{13} + l_{23} \otimes (x_2 \wedge x_3) \otimes s'_{23} \big).
\end{align*}
Hence $f \in \mathrm{Im}(j_n)$ and sequence $(**)$ is exact in the middle.

Finally, we have to show that sequence $(**)$ is exact in the middle when the fibre of $\pi\colon X \to S$ over the point under consideration is unigonal. In this case, the fibre of the $\calO_S$-algebra $\calA$ over this point is isomorphic to 
\[\frac{\C[x_1,x_2,x_3,y]}{(x_1^2 - x_2(ax_2 + bx_3))} = \left(\frac{\C[x_1,x_2,x_3]}{(x_1^2 - x_2(ax_2 + bx_3))}\right)[y]\]
for some $a,b \in \C$.

Once again, let $f$ denote an element of the kernel of the map to $\calA_{2n+1}$. Then, since the $x_i$ form a basis for $\calE_i$, we may write $f = x_1 \otimes f_1 + x_2 \otimes f_2 + x_3 \otimes f_3$ and, using the above characterisation of the fibres of $\calA$, without loss of generality we can replace $f_1$, $f_2$ and $f_3$ with their coefficients in \makebox{$\C[x_1,x_2,x_3]/(x_1^2 - x_2(ax_2 + bx_3))$}. Then the remainder of the proof proceeds much as in the hyperelliptic case. Full details may be found in the proof of \cite[Proposition 4.2.4]{mythesis}.  \end{proof}

The exact sequences $(*)$ and $(**)$ in Proposition \ref{Aseqs} allow us to describe $\calA_{\mathrm{even}}$ as a quotient algebra of $\Sym(\calE_2)$ and $\calA_{\mathrm{odd}}$ as an $\calA_{\mathrm{even}}$-module. The multiplication map $\calA_{\mathrm{odd}} \times \calA_{\mathrm{odd}} \to \calA_{\mathrm{even}}$ is induced by the composition
\[ \calE_1 \otimes \calE_1 \stackrel{\mu_{1,1}}{\longrightarrow} \Sym^2(\calE_1) \stackrel{\sigma_2}{\longrightarrow} \calE_2.\] 
Thus, $\calA$ is completely determined as an $\calO_S$-algebra by the locally free sheaves $\calE_1$ and $\calE_2$ and the map $\sigma_2\colon \Sym^2(\calE_1) \to \calE_2$. 

The structure of $\calA_{\mathrm{even}}$ as a quotient algebra of $\Sym(\calE_2)$ gives a Veronese embedding of $\mathbf{Proj}_S(\calA)$ into $\Proj_S(\calE_2)$ that commutes with the projection to $S$. The projective space bundle $\Proj_S(\calE_2)$ comes equipped with natural invertible sheaves $\calO(n)$ for all $n \in \Z$, which induce invertible sheaves $\calO_{\mathbf{Proj}_S(\calA)}(2n)$ on $\mathbf{Proj}_S(\calA)$.

Now that we have a way to construct $\calA$, we would like to find a way to reconstruct $\calR(X,\pi,\calL)$ from it. By Lemma \ref{Rprops}, we can already construct $\calR(X,\pi,\calL)$ as an \mbox{$\calA$-module}. However, we need to give $\calR(X,\pi,\calL)$ a multiplicative structure to make it into an $\calA$-algebra. In order to do this, we need to determine the multiplication map from $\calE_3^+ \otimes \calE_3^+$ to $\calE_6$. By Lemma \ref{Rprops}, this multiplication map has image contained in $\calA_6$. So the ring structure on $\calR(X,\pi,\calL)$ induces a map
\[\beta\colon (\calE_3^+)^2 \longrightarrow \calA_6.\]

To determine $\beta$, we will study the map $\psi\colon X^c \to \mathbf{Proj}_S(\calA)$. First, however, we need a definition.

\begin{defn} \label{Pdefn} Let $P$ be a point in the support of $\tau$. The fibre of $\mathbf{Proj}_S(\calA)$ over $P$ is of the form
\[\{x_1^2 - x_2(ax_2 +bx_3) = 0\} \subset \Proj_{(1,1,1,2)}[x_1,x_2,x_3,y].\]
This is a cone over the rational normal curve of degree $4$ and is singular at the point $(0\!:\!0\!:\!0\!:\!1)$.

Taking all such singular points associated to the points of $\mathrm{Supp}(\tau)$, we get a subset of $\mathbf{Proj}_S(\calA)$ that we will denote by $\calP$. Note that the projection onto $S$ maps $\calP$ bijectively onto $\mathrm{Supp}(\tau)$.
\end{defn}

Then we have the following analogue of \cite[Theorem 4.7]{flgi}:

\begin{prop} \label{branchdiv} $X^c = \mathbf{Proj}_S(\calR(X,\pi,\calL))$ is a double cover of $\mathbf{Proj}_S(\calA)$, with branch locus consisting of the set of isolated points $\calP$ together with the divisor $B_{\calA}$ in the linear system $|\calO_{\mathbf{Proj}_S(\calA)}(6) \otimes \pi_{\calA}^*(\calE_3^+)^{-2}|$ determined by $\beta$ ($B_{\calA}$ is thus disjoint from $\calP$).
\end{prop}
\begin{proof} (Following the proof of \cite[Theorem 4.7]{flgi}). Note first that $\psi\colon X^c \to \mathbf{Proj}_S(\calA)$ is a double cover by Lemma \ref{Rprops}. It just remains to calculate the branch locus of $\psi$.

Since the question is local on $S$, we may use the same method as in the proof of Lemma \ref{Rprops} and restrict our attention to an affine open set $U$ over which $X^c$ is isomorphic to the subscheme of $\Proj_{(1,1,1,2,3)}[x_1,x_2,x_3,y,z] \times U$ defined by the equations
\[ f_2(x_1,x_2,x_3,y;t) = 0, \makebox[3em]{} z^2 = f_6(x_1,x_2,x_3,y;t),\]
where $t$ is a parameter on $U$. Furthermore, we note that if the $x_i$'s and $y$ simultaneously vanish then $z=0$ also, which is impossible.

At a point where $x_i \neq 0$ for some $i$, we can localise both equations by dividing by $x_i^2$, respectively by $x_i^6$. Then $z = 0$ is the ramification divisor and $f_6 = 0$ is the branch locus. This equation defines exactly the divisor $B_{\calA} \subset \mathbf{Proj}_S(\calA)$.

At a point where $x_1 = x_2 = x_3 = 0$, we may assume that $y =1$ and we have a point of $\calP$. Note that, since the points $(0\!:\!0\!:\!0\!:\!1\!:\!a)$ and $(0\!:\!0\!:\!0\!:\!1\!:\!-a)$ are identified in $\Proj(1,1,1,2,3)$ for any $a \in \C$, this point must be a branch point of $\psi$. Furthermore, by Assumption \ref{mainass}(iv), $f_6$ cannot vanish at such a point, so $B_{\calA}$ is disjoint from $\calP$. \end{proof}

Putting the results of this section together, we can list the data required to construct the relative log canonical model of a threefold fibred by K3 surfaces of degree two. 

\begin{defn} \label{5tupledefn} Let $(X,\pi,\calL)$ be a threefold fibred by K3 surfaces of degree two over a nonsingular curve $S$ that satisfies Assumptions \ref{mainass}. Then define the \emph{associated $5$-tuple of $(X,\pi,\calL)$} over $S$, denoted $(\calE_1,\tau,\xi,\calE_3^+,\beta)$, as follows:
\begin{itemize}
\item $\calE_1 = \pi_*(\omega_X \otimes \calL)$.
\item $\tau$ is the effective divisor on $S$ whose structure sheaf is isomorphic to $\calT_2$.
\item $\xi \in \mathrm{Ext}_{\calO_S}^1(\calO_{\tau}, \Sym^2(\calE_1))/\mathrm{Aut}_{\calO_s}(\calO_{\tau})$ is the isomorphism class of the pair $(\calE_2,\sigma_2)$ in the sequence
\[ 0 \longrightarrow \Sym^2(\calE_1) \stackrel{\sigma_2}{\longrightarrow} \calE_2 \longrightarrow \calO_{\tau} \longrightarrow 0.\]
\item $\calE_3^+$ is the $\iota$-invariant part of $\pi_*((\omega_X \otimes \calL)^{[3]})$.
\item $\beta \in \Proj(H^0(S, \calA_6 \otimes (\calE_3^+)^{-2})) \cong |\calO_{\mathbf{Proj}_S(\calA)}(6) \otimes \pi_{\calA}^*(\calE_3^+)^{-2}|$ is the class of a section with associated divisor $B_{\calA}$.
\end{itemize}
\end{defn}

\begin{remark} We need one more piece of data than Catanese and Pignatelli \cite{flgi}: the line bundle $\calE_3^+$. This is because a surface fibred by genus two curves is naturally polarised by its canonical divisor, but we require a separate polarisation sheaf $\calL$. The extra piece of data in our case is needed to determine this polarisation sheaf.
\end{remark}

\section{A Generality Result}\label{genres}

In this section we will give a method, based upon the results of Section \ref{constructingR}, to construct relative log canonical models of threefolds fibred by K3 surfaces of degree two, and prove a result about the generality of this construction.

Fix a nonsingular complex curve $S$. We begin with a $5$-tuple $(\calE_1,\tau,\xi,\calE_3^+,\beta)$ of data on $S$, defined by:
\begin{itemize}
\item $\calE_1$ is a rank $3$ vector bundle on $S$.
\item $\tau$ is an effective divisor on $S$.
\item $\xi \in \mathrm{Ext}_{\calO_S}^1(\calO_{\tau}, \Sym^2(\calE_1))/\mathrm{Aut}_{\calO_s}(\calO_{\tau})$ yields a pair $(\calE_2,\sigma_2)$ consisting of a vector bundle $\calE_2$ on $S$ and a map $\sigma_2\colon \Sym^2(\calE_1) \to \calE_2$.
\item $\calE_3^+$ is a line bundle on $S$.
\item $\beta \in \Proj(H^0(S, \calA_6 \otimes (\calE_3^+)^{-2}))$, where $\calA_6$ is defined using $\calE_1$, $\calE_2$, $\sigma_2$ and the exact sequences of Proposition \ref{Aseqs}.
\end{itemize}

Given this data, we begin by constructing a sheaf of $\calO_S$-algebras $\calA$ using the exact sequences of Proposition \ref{Aseqs}. Then we may define a second sheaf of $\calO_S$-algebras
\[ \calR := \calA \oplus (\calA[-3] \otimes \calE_3^+),\]
with multiplicative structure induced by $\calA$ and the map $(\calE_3^+)^2 \to \calA_6$ defined by $\beta$.

\begin{defn} \label{admissdefn} We say that a $5$-tuple $(\calE_1,\tau,\xi,\calE_3^+,\beta)$ is \emph{admissible} if the sheaf of algebras $\calR$ constructed from it satisfies the following conditions:
\begin{enumerate}[(i)]
\item Let $B_{\calA}$ be the divisor of $\beta$ on $\mathbf{Proj}_S(\calA)$; then $B_{\calA}$ does not contain any point of the set $\calP$ defined in \ref{Pdefn}.
\item $\mathbf{Proj}_S(\calR)$ has at worst canonical singularities. 
\end{enumerate}
\end{defn}

We have the following generality result, analogue of \cite[Theorem 4.13]{flgi}:

\begin{thm} \label{relmodthm} Fix a nonsingular complex curve $S$. Let $(X,\pi,\calL)$ be a threefold fibred by K3 surfaces of degree two over $S$ that satisfies Assumptions \ref{mainass}. Then the associated $5$-tuple of $(X,\pi,\calL)$ over $S$ is admissible.

Conversely, let $\calR$ be a sheaf of $\calO_S$-algebras defined by an admissible $5$-tuple $(\calE_1,\tau,\xi,\calE_3^+,\beta)$. Let $X = \mathbf{Proj}_S(\calR)$ and $\pi\colon X \to S$ be the natural projection. Then there is a canonically defined polarisation sheaf $\calL$ on $X$ that makes $(X,\pi,\calL)$ into a threefold fibred by K3 surfaces of degree two that satisfies Assumptions \ref{mainass} and, furthermore, $\mathbf{Proj}_S(\calR)$ is the relative log canonical algebra of $(X,\pi,\calL)$ over $S$ and $(\calE_1,\tau,\xi,\calE_3^+,\beta)$ is its associated $5$-tuple.
\end{thm}

\begin{proof} We begin by letting $(X,\pi,\calL)$ be a threefold fibred by K3 surfaces of degree two over $S$ that satisfies Assumptions \ref{mainass}. We want to show that the associated $5$-tuple of $(X,\pi,\calL)$ is admissible. Note that condition (i) in the definition of admissible follows immediately from Proposition \ref{branchdiv}.

It remains to show that $X^c := \mathbf{Proj}_S(\calR)$ has at worst canonical singularities. This will follow from Assumption \ref{mainass}(ii), that the pair $(X,H)$ is canonical. By the proof of Lemma \ref{fingen}, the relative log canonical models of $(X,\pi,\calL)$ and $(X,\pi,\calO_X(H))$ are isomorphic, so it is enough to prove that the relative log canonical model $\hat{X}^c$ of $(X,\pi,\calO_X(H))$ has at worst canonical singularities. 

Let $\hat{\phi}\colon X - \to \hat{X}^c$ denote the natural birational map. By Assumption \ref{mainass}(i) $H$ is irreducible, so no components of $H$ can be contracted by $\hat{\phi}$. Thus, by \cite[Proposition 3.51]{bgav} we see that $\mathrm{discrep}(\hat{X}^{c},\hat{\phi}_+H) \geq \mathrm{discrep}(X,H) \geq 0$, so the log pair $(\hat{X}^{c},\hat{\phi}_+H)$ is canonical. But $\hat{\phi}_+H$ is effective on $\hat{X}^{c}$ so, by \cite[Corollary 2.35]{bgav}, the log pair $(\hat{X}^{c},0)$ is also canonical and $\hat{X}^{c}$ has at worst canonical singularities. This proves condition (ii) in the definition of admissible.

Next we prove the converse statement. Let $\calR$ be a sheaf of $\calO_S$-algebras defined by an admissible $5$-tuple $(\calE_1,\tau,\xi,\calE_3^+,\beta)$. Define $X := \mathbf{Proj}_S(\calR)$ and let $\pi\colon X \to S$ denote the natural projection. As in the proof of Lemma \ref{Rprops}, over an affine open set $U \subset S$ we can view $X$ as a normal subvariety in $\Proj_{(1,1,1,2,3)}[x_1,x_2,x_3,y,z] \times U$ defined by equations
\[ f_2(x_1,x_2,x_3,y;t) = 0, \makebox[3em]{} z^2 = f_6(x_1,x_2,x_3,y;t),\]
where $t$ is a parameter on $U$.

By this local description, it is clear that $X$ does not intersect the singular curve $(0\!:\!0\!:\!0\!:\!0\!:\!1) \times U \subset \Proj{(1,1,1,2,3)} \times U$. Furthermore, the divisor $B_{\calA}$ defining the locus $f_6(x_1,x_2,x_3,y;t)=0$ on $\mathbf{Proj}_S(\calA)$ does not intersect the set $\calP$ by condition (ii) in the definition of admissible, so $X$ does not intersect the singular curve \mbox{$(0\!:\!0\!:\!0\!:\!1\!:\!0) \times U \subset \Proj{(1,1,1,2,3)} \times U$} either.

Thus, we see that $X$ does not intersect the singular locus in $\Proj{(1,1,1,2,3)} \times U$ and that $\omega_X|_{\pi^{-1}(U)}$ is trivial. Therefore $X$ is Gorenstein and over each such open set the sheaf $\calO(1)$ induced on $X$ by the weighted projective space structure is invertible. These invertible sheaves glue to give a invertible sheaf $\calO_{X}(1)$ on $X$.

Next consider the morphism $\pi\colon X \to S$. By construction $\pi$ is flat, projective and surjective. Furthermore, by the local description above, the general fibre of $\pi$ is a complete intersection of type $(2,6)$ in $\Proj{(1,1,1,2,3)}$, which is a K3 surface by adjunction. As $X$ has at worst canonical singularities, by \cite[Lemma 5.17]{bgav} we see that this K3 surface has at worst Du Val singularities. 

Now define $\calL := \calO_X(1) \otimes \omega_X^{-1}$. Then $\calL$ is invertible and the local description above shows that $\calL$ induces an ample invertible sheaf with self-intersection number two on a general fibre of $\pi$. Therefore $(X,\pi,\calL)$ is a threefold fibred by K3 surfaces of degree two.

We next show that Assumptions \ref{mainass} hold for $(X,\pi,\calL)$.

\begin{lemma}\label{asslem}  Define a threefold fibred by K3 surfaces of degree two $(X,\pi,\calL)$ as above. Then $(X,\pi,\calL)$ satisfies Assumptions \ref{mainass}.\end{lemma}

\begin{proof} We begin by proving Assumptions \ref{mainass}(iii) and \ref{mainass}(iv). Let $U \subset S$ be an affine open set. Then as in the proof of Lemma \ref{Rprops}, we can view $\pi^{-1}(U)$ as a complete intersection
\[\pi^{-1}(U) \cong \{ f_2(t) = f_6(t) = 0\} \subset \Proj(1,1,1,2,3) \times U.\]
As $\omega_X|_{\pi^{-1}(U)}$ is trivial, the restriction of $\calL$ to $\pi^{-1}(U)$ is just the sheaf induced from $\calO(1)$ on $\Proj(1,1,1,2,3) \times U$. From this, it easily follows that Assumptions \ref{mainass}(iii) and \ref{mainass}(iv) hold for $(X,\pi,\calL)$.

It remains to prove Assumptions \ref{mainass}(i) and \ref{mainass}(ii). In order to do this we start by defining the divisor $H$. Choose an ample invertible sheaf $\calN$ on $S$. Then for some $m > 0$, the sheaf $\pi_*\calL \otimes \calN^m$ is generated by its global sections. Let $H$ denote a general member of the linear system $|\calL \otimes \pi^* \calN^m|$. 

We wish to show that $H$ is a prime divisor that is flat over $S$, and that $(X,H)$ is canonical. In order to do this, we show that $H$ may be chosen to avoid the worst singularities of $Y$. Specifically, we want to avoid singularities that are not compound Du Val \cite[Definition 5.32]{bgav}.

We start by examining the linear system $|\calL \otimes \pi^* \calN^m|$ in which $H$ moves. Note that as $\pi_*\calL \otimes \calN^m$ is generated by its global sections, for any affine open set $U \subset S$ the sections in $H^0(X, \calL \otimes \pi^* \calN^m)$ generate $H^0(\pi^{-1}(U), \calL \otimes \pi^* \calN^m)$ as an $\calO_{\pi^{-1}(U)}$-module, so we may study this linear system locally over $S$.

So let $U \subset S$ be an affine open set. As $\calN|_U \cong \calO_U$, we have the inverse image $\pi^* \calN^m|_{\pi^{-1}(U)} \cong \calO_{\pi^{-1}(U)}$ and, since $\omega_X|_{\pi^{-1}(U)}$ is trivial, the restriction of $\calL \otimes \pi^* \calN^m$ to $\pi^{-1}(U)$ is just the sheaf induced from $\calO(1)$ on $\Proj(1,1,1,2,3) \times U$. 

The sheaf $\calO(1)$ defines a linear system on $\Proj(1,1,1,2,3) \times U$ that is base point free outside of the locus $(0\!:\!0\!:\!0\!:\!y\!:\!z) \times U$, so the induced linear system on $X$ is base point free outside of its intersection with this locus, consisting of precisely the points over the set $\calP$ defined in \ref{Pdefn}. Thus, as $H^0(\pi^{-1}(U), \calL \otimes \pi^* \calN^m)$ is generated as an $\calO_{\pi^{-1}(U)}$-module by the sections in $H^0(X, \calL \otimes \pi^* \calN^m)$, we see that the linear system $|\calL \otimes \pi^* \calN^m|$ has no base points or fixed components on $X$ outside of the points over the set $\calP$. In particular, we see that the linear system induced by $|H|$ on a general fibre of $\pi$ is base point free.

As $X$ has at worst canonical singularities, by \cite[Corollary 5.40]{bgav} all but finitely many of the singular points of $X$ are compound Du Val. So, apart from the points lying over the set $\calP$, we may assume that the only singularities of $X$ lying on $H$ are compound Du Val. Furthermore, by Bertini's theorem we may assume that $H$ is reduced, irreducible and nonsingular outside of the singular points of $X$ and the points lying over $\calP$. In particular $H$ cannot contain any components of fibres, so is horizontal and thus flat over $S$. This proves that $H$ is flat over $S$ and that $(X,\pi,\calL)$ satisfies Assumption \ref{mainass}(i) (with $\calM = \calN^{-m}$).

The last step in the proof of Lemma \ref{asslem} is to show that, with $H$ chosen as above, the log pair $(X,H)$ is canonical. This will follow from \cite[Theorem 5.34]{bgav} if we can show that all of the singularities in $H$ are rational double points. By the argument above, these singularities arise from compound Du Val points and points lying over $\calP$. At a compound Du Val point, the singularity in $H$ is a rational double point by definition. So it just remains to classify the singularities lying over the points of $\calP$.

By the proof of Lemma \ref{Tprops}, after a change of coordinates locally we can write $\mathbf{Proj}_S(\calA)$ as 
\[\{f_2(t)=0\} \subset \Proj_{(1,1,1,2)}[x_1,x_2,x_3,y] \times U,\]
where 
\[f_2(t) = x_1^2 - x_2(ax_2 + bx_3) + t^r y + t \psi(x_i,t)\]
for some $a,b \in \C$ that are not both zero and $t$ a local parameter on the affine open set $U \subset S$. The weighted projective space structure induces a divisorial sheaf $\calO_{\mathbf{Proj}_S(\calA)}(1)$ locally on $\mathbf{Proj}_S(\calA)$, a general section of which defines a Weil divisor that has a rational double point singularity of type $A_{2r+1}$ at the point $(0\!:\!0\!:\!0\!:\!1\,;0)$. As $B_{\calA}$ does not contain any point of $\calP$, around the point $(0\!:\!0\!:\!0\!:\!1\,;0)$ we have that $X$ is a cyclic double cover of $\mathbf{Proj}_S(\calA)$ ramified over the point $(0\!:\!0\!:\!0\!:\!1\,;0)$. Thus a divisor defined by a general section of $\calO_{X}(1)$ is a cyclic double cover of a divisor defined by a general section of $\calO_{\mathbf{Proj}_S(\calA)}(1)$ ramified over the singularity. Therefore, by \cite[Theorem 5.43]{bgav}, the general section of $\calO_{X}(1)$ has a rational double point singularity of type $A_{r}$.

Thus, we may assume that the only singularities occurring in $H$ are rational double points, so by \cite[Theorem 5.34]{bgav} the log pair $(X,H)$ is canonical. Therefore Assumption \ref{mainass}(ii) holds for $(X,\pi,\calL)$. This completes the proof of Lemma \ref{asslem}. \end{proof}

To complete the proof of Theorem \ref{relmodthm}, we just have to show that $\mathbf{Proj}_S(\calR)$ is the relative log canonical model of $(X,\pi,\calL)$ over $S$ and $(\calE_1,\tau,\xi,\calE_3^+,\beta)$ is its associated $5$-tuple. This will follow if we can show that $\calR$ is the relative log canonical algebra of $(X,\pi,\calL)$.

Note that
\[ \pi_*((\omega_X \otimes \calL)^{[n]}) \cong \pi_*((\omega_X \otimes \calO_X(1) \otimes \omega_X^{-1})^{[n]}) \cong \pi_*(\calO_X(1)^{n})\]
for all $n > 0$. But this implies that the relative log canonical algebra of $(X,\pi,\calL)$ is $\calR$, as required. This completes the proof of Theorem \ref{relmodthm}.\end{proof}

\proof[Acknowledgements] Many of the results in this work are extensions of results that first appeared in my doctoral thesis \cite{mythesis}, completed at the University of Oxford. I would like to thank my doctoral advisor Bal\'{a}zs Szendr\H{o}i for his support and guidance throughout the writing of this paper. I would also like to thank Miles Reid for a helpful conversation in which he directed my attention to the paper \cite{flgi} that inspired many of the results in this work, and Charles Doran for suggesting several of the extensions to the material in \cite{mythesis} that appear here.

\end{document}